\documentclass[11 pt]{article}
\usepackage{amssymb,latexsym,amsmath}
\usepackage{amsthm, authblk, cite}
\usepackage{graphicx}
\usepackage{relsize}

\theoremstyle{plain}
\newtheorem{theorem}{Theorem}

\newtheorem{proposition}[theorem]{Proposition}
\newtheorem{lemma}[theorem]{Lemma}

\theoremstyle{definition}
\newtheorem{definition}[theorem]{Definition}
\newtheorem{example}[theorem]{Example}

\newtheorem{remark}[theorem]{Remark}

\newtheorem{algorithm}{Algorithm}

\usepackage{hyperref}

\newcommand{\OO}{{\mathcal O}}
\usepackage[margin=3cm]{geometry}

\newcommand{\FF}{{\mathbb F}}
\DeclareMathOperator{\res}{res}
\DeclareMathOperator{\divv}{div}
\DeclareMathOperator{\im}{Im}

\title{Compression for trace zero points on twisted Edwards curves}

\author{Giulia Bianco and Elisa Gorla\thanks{The research reported in this paper was partially supported by the Swiss National Science Foundation under grant no. 200021\_150207.}}

\affil{Institut de Math\'{e}matiques, Universit\'{e} de Neuch\^{a}tel\\Rue Emile-Argand 11, CH-2000 Neuch\^{a}tel, Switzerland}

\date{}

\begin{document}
\maketitle

\begin{abstract}
We propose two optimal representations for the elements of trace zero subgroups of twisted Edwards curves. 
For both representations, we provide efficient compression and decompression algorithms. The efficiency of the algorithm 
is compared with the efficiency of similar algorithms on elliptic curves in Weierstrass form.
\end{abstract}

\section*{Introduction}\label{intr0}

Trace zero subgroups are subgroups of the groups of points of an elliptic curve over extension fields. They were first proposed for use in public key cryptography by Frey in \cite{Frey}. A main advantage of trace zero subgroups is that they offer a better scalar multiplication performance than the whole group of points of an elliptic curve of approximately the same cardinality. This allows a fast arithmetic, which can speed up the calculations by 30\% compared with elliptic curves groups (see e.g. \cite{langetzero} for the case of hyperelliptic curves, \cite{ac07} and \cite{cesenathesis} for elliptic curves over fields of even characteristic). In addition, computing the cardinality of a trace zero subgroup is more efficient than for the group of points of an elliptic curve of approximately the same cardinality. Moreover, the DLP in a trace zero subgroup has the same complexity as the DLP in the group of $\FF_{q^n}$-rational points of the curve, of which the trace zero subgroup is a proper subgroup. Hence, when we restrict to this subgroup, we gain a more efficient arithmetic without compromising the security. Finally, in the context of pairings trace zero subgroups of supersingular elliptic curves offer higher security than supersingular elliptic curves of the same bit-size, as shown in~\cite{rs09}.

The problem of how to compress the elements of the trace zero subgroup is the analogue within elliptic (and hyperelliptic) curve cryptography of torus-based cryptography in finite fields. For elliptic and hyperelliptic curves this problem has been studied by many authors, see \cite{nau99}, \cite{langetzero}, \cite{sil05}, \cite{rs09}, \cite{EM1}, and \cite{EM2}.

Edwards curves were first introduced by H.M. Edwards in \cite{EdEd} as a normal form for elliptic curves. They were proposed for use in elliptic curve cryptography by Bernstein and Lange in \cite{BL1}. Twisted Edwards curves were introduced shortly after in \cite{BLal}. They are relevant from a cryptographic point of view since the group operation can be computed very efficiently and via strongly unified formulas, i.e. formulas that do not distinguish between addition and doubling. This makes them more resistant to side-channel attacks. We refer to \cite{BL1}, \cite{BL2}, and \cite{BLal} for a detailed discussion of the advantages of Edwards curves. 

In this paper, we provide two efficient representations for the elements of the trace zero subgroups of twisted Edwards curves. The first one follows ideas from \cite{EM1} and it is based on Weil restriction of scalars and Semaev's summation polynomials. The second one follows ideas from \cite{EM2} and it makes use of rational functions on the curve. Some obstacles have to be overcome in adapting these ideas to Edwards curves, especially for adapting the method from~\cite{EM2}.  

Given a twisted Edwards curve defined over a finite field $\FF_q$ of odd characteristic and a field extension of odd prime degree $\FF_q \subset \FF_{q^n}$, we consider the trace zero subgroup $\mathcal{T}_n$ of the group of $\FF_{q^n}$-rational points of the curve. We give two efficiently computable maps from $\mathcal{T}_n$ to $\FF_{q}^{n-1}$, such that inverse images can also be efficiently computed. One of our maps identifies Frobenius conjugates, while the other identifies Frobenius conjugates and negatives of points. Since $\mathcal{T}_n$ has order $\OO(q^{n-1})$, our maps are optimal representations of $\mathcal{T}_n$ modulo Frobenius equivalence. For both representations we provide efficient algorithms to calculate the image and the preimage of an element, that is, to compress and decompress points. We also compare with the corresponding algorithms for trace zero subgroups of elliptic curves in short Weierstrass form.

The article is organized as follows: In Section 1 we give some preliminaries on twisted Edwards curves, finite fields, trace zero subgroups, and representations. In Section 2 we present our first optimal representation based on Weil restriction and summations polynomials, and give compression and decompression algorithms. We then make explicit computations for the cases $n=3$ and $n=5$, and compare execution times of our Magma implementation with those of the corresponding algorithms for elliptic curves in short Weierstrass form. In Section 3 we propose another representation based on rational functions, with the corresponding algorithms, computations, and efficiency comparison.

\section{Preliminaries and notations}

Let $\mathbb{F}_q$ be a finite field of odd characteristic and let $\mathbb{F}_q \subset \mathbb{F}_{q^n}$ be a field extension of odd prime order. Choose a normal basis $\{\alpha,\alpha^q,\dots,\alpha^{q^{n-1}}\}$ of $\mathbb{F}_{q^n}$ 
over $\mathbb{F}_q$. If $n|q-1$, let $\mathbb{F}_{q^n} = \mathbb{F}_q[\xi]/(\xi^n-\mu)$, where $\mu$ is not a $n^{th}$-power in $\mathbb{F}_q$, and choose the basis $\{1,\xi,\dots,\xi^{n-1}\}$ of $\mathbb{F}_{q^n}$ over $\mathbb{F}_q$. This choice is particularly suitable for computation, since it produces sparse equations. When writing explicit formulas, we always assume that we are in the latter situation.

When counting the number of operations in our computations, we denote respectively by M, S, and I multiplications, squarings, and inversions in the field. We do not take into account additions and multiplications by constants. The timings for the implementation of our algorithms in Magma refer to version V2.20-7 of the software, running on a single 3 GHz core.

\subsection{Twisted Edwards curves} 

\begin{definition}
A {\bf twisted Edwards curve} over $\FF_q$ is a plane curve of equation
$$E_{a,d} : ax^2+y^2 = 1 + dx^2y^2,$$
where $a,d \in\FF_q\setminus \{0\}$ and $a\not=d$.
An {\bf Edwards curve} is a twisted Edwards curve with $a=1$.
\end{definition}

Twisted Edwards curves are curves of geometric genus one with two ordinary multiple points, namely the two points at infinity.
Since $E_{a,d}$ is birationally equivalent to a smooth elliptic curve, one can define a group law on the set of points of $E_{a,d}$, called the twisted Edwards addition law.

\begin{definition}
The sum of two points $P_1=(x_1,y_1)$ and $P_2=(x_2,y_2)$ of $E_{a,d}$ is defined as
$$P_1+P_2=(x_1,y_1)+(x_2,y_2)=\left(\frac{x_1y_2+x_2y_1}{1+dx_1x_2y_1y_2},\frac{y_1y_2-ax_1x_2}{1-dx_1x_2y_1y_2}\right).$$
\end{definition}

We refer to \cite[Section 3]{BL1} and \cite[Section 6]{BLal} for a detailed discussion on the formulas and a proof of correctness. 
The point $\OO=(0,1)\in E_{a,d}$ is the neutral element of the addition, and we denote by $-P$ the additive inverse of $P$. If $P=(x,y)$, then $-P=(-x,y)$. We let $\OO'=(0,-1)\in E_{a,d}$, and denote by $\Omega_1=[1,0,0]$ and $\Omega_2=[0,1,0]$ the two points at infinity of $E_{a,d}$.

Edwards curves were introduced in \cite{EdEd} as a convenient normal form for elliptic curves. Over an algebraically closed field, every elliptic curve in Weierstrass form is birationally equivalent to an Edwards curve, and vice versa. This is however not the case over $\FF_q$, where Edwards curves represent only a fraction of elliptic curves in Weierstrass form.
In \cite[Theorem 3.2]{BLal} it is shown that a twisted Edwards curve defined over $\FF_q$ is birationally equivalent over $\FF_q$ to an elliptic curve in Montgomery form, and conversely, an elliptic curve in Montgomery form defined over $\FF_q$ is birationally equivalent over $\FF_q$ to a twisted Edwards curve. Moreover, the twisted Edwards addition law corresponds to the usual addition law on an elliptic curve in Weierstrass form via the birational isomorphism, as shown in~\cite[Theorem 3.2]{BL1}. Similarly to elliptic curves in Weierstrass form, the twisted Edwards addition law has a geometric interpretation. 

\begin{proposition}{\em (\!\!\cite[Section 4]{GeomInt})}\label{conic}
Let $P_1,P_2\in E_{a,d}$, and let $C$ be the projective conic passing through $P_1$, $P_2$, $\Omega_1$, $\Omega_2$, and $\mathcal{O}'$. Then the point $P_1+P_2$ is the symmetric with respect to the $y$-axis of the eighth point of intersection between $E_{a,d}$ and $C$.
\end{proposition}

\subsection{Trace zero subgroups}

Let $E_{a,d}$ be a twisted Edwards curve defined over $\mathbb{F}_q$.
We denote by $E_{a,d}(\mathbb{F}_{q^n})$ the group of $\mathbb{F}_{q^n}$-rational points of $E_{a,d}$, by $P_{\infty}$ any point at infinity of $E_{a,d}$, and by $\varphi$ the Frobenius endomorphism on $E_{a,d}$:
$$\varphi : E_{a,d} \longrightarrow E_{a,d} \mbox{ , } (x,y) \mapsto (x^q,y^q) \mbox{ , } P_{\infty}\mapsto P_{\infty}.$$

\begin{definition}
The \textbf{trace zero subgroup} $\mathcal{T}_n$ of $E_{a,d}(\FF_{q^n})$ is the kernel of the trace map
$${\rm Tr}: E_{a,d}(\FF_{q^n}) \longrightarrow E_{a,d}(\FF_q) \mbox{ , } P \mapsto P + \varphi(P)+\varphi^2(P)+\dots + \varphi^{n-1}(P).$$
\end{definition}

We can view $\mathcal{T}_n$ as the $\FF_{q}$-rational points of an abelian variety of dimension $n-1$ defined over $\FF_q$, called the trace zero variety. We refer to \cite{handbook} for a construction and the basic properties of the trace zero variety.
The following result is an easy consequence of \cite{handbook}, Proposition 7.13.

\begin{proposition}\label{exseq}
The sequence
$$0 \longrightarrow E_{a,d}(\mathbb{F}_q) \longrightarrow E_{a,d}(\FF_{q^n}) \stackrel{\varphi -{\rm id}}{\longrightarrow} \mathcal{T}_n \longrightarrow 0$$
is exact. Therefore the DLPs in $E_{a,d}(\FF_{q^n})$ and in $\mathcal{T}_n$ have the same complexity.
\end{proposition}

\subsection{Representations}

\begin{definition}
Let $G$ be a finite set and $\ell \in \mathbb{Z}_+$. A \textbf{representation} of $G$ of size $\ell$ is a map
$$\mathcal{R}: G \longrightarrow \FF_{2}^{\ell},$$
with the property that an element of $\mathbb{F}_2^{\ell}$ has at most $d$ inverse images, where $d=\OO(1)$.
A representation is \textbf{optimal} if $$\ell = \lceil\log_2{|G|}\rceil+ \OO(1).$$
Given $\gamma \in G$ and $x \in {\rm Im}\mathcal{R}$, we call {\bf compression} and {\bf decompression} the process of computing $\mathcal{R}(\gamma)$ and $\mathcal{R}^{-1}(x)$, respectively.
\end{definition}

\begin{remark}
Define an equivalence relation in $G$ via $g \sim h$ iff $\mathcal{R}(g)=\mathcal{R}(h)$.
Any representation $\mathcal{R}$ of $G$ of size $\ell$ induces an injective representation of $\overline{G}=G/ \sim$ of size $\ell$:
$$\overline{\mathcal{R}}: \overline{G} \longrightarrow \mathbb{F}_2^{\ell}.$$
Since $\log_2 |G|=\log_2{|\overline{G}|}+\OO(1)$, $\mathcal{R}$ is an optimal representation of $G$ if and only if $\overline{\mathcal{R}}$ is an optimal representation of $\overline{G}$. Hence the definition of optimal representation is independent of the constant $d$. 
\end{remark}

\begin{remark}\label{reprFq}
It is well known that $\mathbb{F}_q$ has an optimal representation of size $\lceil\log_{2}q\rceil$. Therefore, if $|G|=\Theta(q^m)$, an optimal representation of $G$ may be given via
\begin{equation}\label{optrepr}\mathcal{R}: G \longrightarrow \mathbb{F}_{q}^m \times \mathbb{F}_2^k,\end{equation}
where $k=\OO(1)$.
\end{remark}

In this paper we give two representations of $\mathcal{T}_n$ with $m=n-1$ and $d=n$ or $d=2n$. They are optimal, since 
$|\mathcal{T}_n|=\Theta(q^{n-1})$ by Proposition~\ref{exseq}. 

\section{An optimal representation using summation polynomials}

Let $\FF_q$ be a finite field of odd characteristic and let $E_{a,d}$ be the twisted Edwards curve of equation 
$$ax^2+y^2=1+dx^2y^2$$ where $a,d\in\FF_q\setminus \{0\}$ and $a\neq d$. 
Following ideas from \cite{EM1}, in this section we use Weil restriction of scalars
and Semaev's summation polynomials to write an equation for the subgroup $\mathcal{T}_n$. 
Similarly to the case of elliptic curves in Weierstrass form, a point $P=(x,y)\in E_{a,d}(\FF_{q^n})$ can be represented via $y\in\FF_{q^n}$. Using the curve equation, the value of $x$ can be recovered up to sign.  Hence, after choosing an $\FF_q$-basis of $\FF_{q^n}$, each pair of points $\pm P\in E_{a,d}(\FF_{q^n})$ can be represented by the element $(y_0,\ldots,y_{n-1})\in\FF_q^n$ corresponding to $y\in\FF_{q^n}$ under the isomorphism $\FF_{q^n}\cong\FF_q^n$ induced by the chosen basis. Having an equation for $\mathcal{T}_n$ allows us to drop one of the $y_i$'s and represent each pair $\pm P$ via $n-1$ coordinates in $\FF_q$, thus providing an optimal representation for the elements of $\mathcal{T}_n$. In order to make computation of the compression and decompression maps more efficient, we modify this basic idea and use the elementary symmetric functions of $y,y^q,\ldots,y^{q^{n-1}}$ instead of the vector $(y_0,\ldots,y_{n-1})\in\FF_q^n$.

Summation polynomials were introduced by Semaev in~\cite{SemPol} for elliptic curves in Weierstrass form. Here
we use them in the form for Edwards curves from~\cite{SemPolEd}.

\begin{definition}
The n-th {\bf summation polynomial} is denoted by $f_n$ and defined recursively by
$$f_3(z_1,z_2,z_3)=(z_1^2z_2^2-z_1^2-z_2^2+a{d^{-1}})z_3^2+2(d-a){d^{-1}}z_1z_2z_3+a{d^{-1}}(z_1^2+z_2^2-1)-z_1^2z_2^2,$$
$$f_n(z_1,\dots,z_n)=\res_{t}(f_{n-k}(z_1,\dots,z_{n-k-1},t),f_{k+2}(z_{n-k},\dots,z_n,t))$$ for all $n\geq 4$ and for all $1\leq k \leq n-3$, where $\res_{t}(f_i,f_j)$ denotes the resultant of $f_i$ and $f_j$ with respect to $t$.
\end{definition}

The next theorem summarizes the properties of summation polynomials.

\begin{theorem}[\cite{SemPol} Section 2 and \cite{SemPolEd} Section 2.3.1]

Let $n\geq 3$, let $f_n\in\FF_q[z_1,\ldots,z_n]$ be the n-th summation polynomial. Denote by $\FF_q\subset k$ a field extension, and by $\overline{k}$ its algebraic closure. Then:\begin{enumerate}
\item $f_n$ is absolutely irreducible, symmetric, and has degree $2^{n-2}$ in each of the variables. 
\item $(\beta_1,\ldots,\beta_n)\in k^n$ is a root of $f_n$ if and only if there exist $\alpha_1,\ldots,\alpha_n\in\overline{k}$ such that $P_i=(\alpha_i,\beta_i)\in E_{a,d}(\overline{k})$ and $P_1+\ldots+P_n=\OO$.
\end{enumerate}
\end{theorem}

By the previous theorem, if $P=(x,y)\in \mathcal{T}_n$, then
\begin{equation}\label{repr_equation}
f_n(y,y^q,\dots,y^{q^{n-1}})=0.\end{equation}
A partial converse and exceptions to the opposite implication are given in the next proposition.

\begin{proposition}{\em(\!\!\cite[Lemma 1 and Proposition 4]{EM1})}\label{exc}
Let $E_{a,d}$ be a twisted Edwards curve and denote by $E_{a,d}[m]$ its m-torsion points. We have:
\begin{itemize}
\item[\rm (1)] $\mathcal{T}_3 = \{(x,y) \in E_{a,d}(\mathbb{F}_{q^3})\mid f_3(y,y^q,y^{q^2})=0\}$,
\item[\rm (2)] $\mathcal{T}_5\cup E_{a,d}[3](\mathbb{F}_q) = \{(x,y) \in E_{a,d}(\mathbb{F}_{q^5})\mbox{ }  | \mbox{ } f_5(y,y^q,\dots,y^{q^4})=0\}$,
\item[\rm (3)] $\mathcal{T}_n\cup\bigcup_{k=1}^{\lfloor \frac{n}{2} \rfloor}E_{a,d}[n-2k](\mathbb{F}_q) \subseteq \{ (x,y) \in E_{a,d}(\mathbb{F}_{q^n}) \mid f_n(y,y^q,\ldots,y^{q^{n-1}})=0\}$ for $n\geq 7$.
\end{itemize}
\end{proposition}

\begin{proof} The proof proceeds as in Lemma 1 and Proposition 4 of \cite{EM1}, after observing that for any odd prime $n$ one has $E_{a,d}[2]\cap \mathcal{T}_n = \{ \mathcal{O} \}$.
\end{proof}

\begin{remark}\label{exc_rmk}
Proposition~\ref{exc} raises the question of efficiently deciding, for each root $y\in\FF_{q^n}$ of equation (\ref{repr_equation}), whether the corresponding points $(\pm x,y)\in E_{a,d}$ are elements of $\mathcal{T}_n$. However, this issue is easily solved in the two cases of major interest $n=3$ and $n=5$. In fact:
\begin{itemize}
\item By Proposition~\ref{exc} (1), $(\pm x,y)\in \mathcal{T}_3$ if and only if $x\in \mathbb{F}_{q^3}$.
\item By Proposition~\ref{exc} (2), $(\pm x,y)\in \mathcal{T}_5$ if and only if $x \in \mathbb{F}_{q^5}$ and $(\pm x,y)\not \in E_{a,d}[3](\mathbb{F}_q)\setminus \{\mathcal{O}\}$. By storing the list $\mathcal{L}$ of the \textit{y}-coordinates of the elements of $E_{a,d}[3](\FF_q)\setminus\{\OO\}$, one can easily decide whether a point of $E_{a,d}(\FF_{q^5})$ of coordinates $(x,y)$ belongs to $\mathcal{T}_n$ by checking that $y \not\in\mathcal{L}$. Notice that $\mathcal{L}$ consists of at most 4 elements of $\FF_q$.
\end{itemize} 
\end{remark}

Using the above considerations as a starting point, we can give an optimal representation for the points of $\mathcal{T}_n$ with efficient compression and decompression algorithms.

\medskip 

{\bf 1.} Denote by $e_1,\ldots,e_n$ the elementary symmetric functions in $n$ variables.
Represent $(x,y)\in \mathcal{T}_n$ via $n-1$ of the elementary symmetric functions evaluated at $y,y^q,\ldots,y^{q^{n-1}}$.
We obtain an efficiently computable optimal representation
\begin{equation}\label{repr_sem}
\begin{array}{rccl}
\mathcal{R} : & \mathcal{T}_n & \longrightarrow & \FF_q^{n-1} \\
& (x,y) & \longmapsto &(e_i(y,y^q,\dots,y^{q^{n-1}}))_{i=1,\dots,n-1}.\end{array}\end{equation}
\medskip

{\bf 2.} Since the  polynomial $f_n(z_1,\ldots,z_n)$ is symmetric, we can write it uniquely as a polynomial $g_n(e_1,\dots,e_n)\in\FF_q[e_1,\ldots,e_n]$. 
Therefore, the equation
$$g_n(e_1,\dots,e_n)=0$$
describes trace zero points (with the exceptions seen in Proposition~\ref{exc}) via the equations
\begin{equation}\label{sistsym} e_1=\tilde{e}_1(y_0,\dots,y_{n-1}),\ldots,e_n=\tilde{e}_n(y_0,\dots,y_{n-1}),\end{equation}
where the polynomials $\tilde{e_1},\ldots,\tilde{e_n}$ are obtained from the polynomials
$$e_1(y,y^q,\dots,y^{q^{n-1}}),\ldots,e_n(y,y^q,\dots,y^{q^{n-1}})$$
by Weil restriction of scalars with respect to the chosen basis of $\mathbb{F}_{q^n}$ over $\mathbb{F}_q$, and reducing modulo $y_i^q-y_i$ for $i\in\{0,\ldots,n-1\}$. 
Notice that the reduction simplifies the equations by drastically reducing their degrees. 
Moreover, it does not alter their values when evaluated over $\FF_q$. 

\medskip

{\bf 3.} For $(e_1,\ldots,e_{n-1})\in\mathcal{R}(\mathcal{T}_n)$, we first solve $g_n(e_1,\ldots,e_{n-1},t)=0$ for $t$. For any solution $e_n\in\FF_q$, we solve system (\ref{sistsym}) to find $(y_0,\ldots,y_{n-1})\in\FF_q^n$, corresponding to $y\in\FF_{q^n}$. From $y$ we can recover $x$ in the usual way (see also Remark~\ref{exc_rmk}).

\bigskip

Notice that $g_n(e_1,\ldots,e_n,)$ is not linear in any of the variables for $n\geq 3$, hence in 3. we may find more than one value for $e_n$. This corresponds to the fact that $\mathcal{R}$ may identify more than just opposites and Frobenius conjugates. 
However this is a rare phenomenon, and for a generic point $P\in\mathcal{T}_n$, $\mathcal{R}^{-1}(\mathcal{R}(P))$ consists only of $\pm P$ and their Frobenius conjugates. We come back to this discussion in Subsection~\ref{n=5}, where we discuss this issue for $n=5$.

We now give the pseudocode of a compression and decompression algorithm for the elements of $\mathcal{T}_n$.
\newpage 

\hrule
\begin{algorithm}[{\bf Compression}]\textbf{ }\\
{
\hrule
\textbf{ }\\
\newline
\textbf{Input} : $P=(x,y) \in \mathcal{T}_n$\\
\textbf{Output} : $\mathcal{R}(P) \in \mathbb{F}_q^{n-1}$\\
\newline
\begin{footnotesize} 1: \end{footnotesize} Write $y=y_0\alpha+\ldots+y_{n-1}\alpha^{q^{n-1}}$.\\
\begin{footnotesize} 2: \end{footnotesize} Compute $e_i = \tilde{e}_i(y_0,\dots,y_{n-1})$ for $i = 1, \dots , n-1$.\\
\begin{footnotesize} 3: \end{footnotesize} {\bf return} $(e_1,\ldots,e_{n-1})$\\
\hrule
}
\end{algorithm}

\textbf{ }\\
\hrule

\begin{algorithm}[{\bf Decompression}]\label{aldecsem}
\textbf{ }\\
{
\hrule
\textbf{ }\\
\newline
\textbf{Input} : $(e_1,\dots,e_{n-1}) \in \mathbb{F}_q^{n-1}$\\
\textbf{Output} : $\mathcal{R}^{-1}(e_1,\dots,e_{n-1})\subseteq\mathcal{T}_n$\\
\newline
\begin{footnotesize}1: \end{footnotesize} Solve $g_n(e_1,\dots,e_{n-1},t)=0$ for $t$ in $\mathbb{F}_q$.\\
\begin{footnotesize}2: \end{footnotesize} $T\leftarrow$ list of solutions of $g_n(e_1,\dots,e_{n-1},t)=0$ in $\mathbb{F}_q$.\\
\begin{footnotesize}3: \end{footnotesize} \textbf{for} $e_n\in T$, find a solution in $\FF_q^n$of the system\\
$\left\{\begin{tabular}{rcl}
$e_1$ & $=$ & $\tilde{e}_1(y_0,\dots,y_{n-1})$ \\
&$\vdots $&\\
$e_n$ &  $=$ & $\tilde{e}_n(y_0,\dots,y_{n-1})$ \\
\end{tabular}\right.$ if it exists.\\
\begin{footnotesize}4: \end{footnotesize} Any time a solution $(y_0,\dots,y_{n-1})$ is found, compute $y=y_0\alpha+\dots+y_{n-1}\alpha^{q^{n-1}}$.\\
\begin{footnotesize}5: \end{footnotesize} Recover one of the corresponding \textit{x}-coordinates using the curve equation.\\
\begin{footnotesize}6: \end{footnotesize}\textbf{ end for}\\ 
\begin{footnotesize}7: \end{footnotesize} \textbf{if} $(x,y) \in \mathcal{T}_n$ \textbf{then}\\
\begin{footnotesize}8: \end{footnotesize} Add $P=(\pm x,y)$ and all its Frobenius conjugates to the list $L$ of output points.\\
\begin{footnotesize}9: \end{footnotesize} \textbf{end if}\\
\begin{footnotesize}10: \end{footnotesize}{\bf return} $L$\\
\hrule
}
\end{algorithm}

\subsection{Explicit equations, complexity, and timings for $n=3$} 

In this subsection we give explicit equations for trace zero point compression and decompression on twisted Edwards curves for $n=3$. We also estimate the number of operations needed for the computations, present some timings obtained with Magma, and compare with the results from~\cite{EM1} for elliptic curves in short Weierstrass form. 
 
The symmetrized third summation polynomial for $E_{a,d}$ is
\begin{equation}\label{g3Ed}
g_3(e_1,e_2,e_3)=e_1^2 - 1 + (d/a)(e_3^2-e_2^2) +(2d/a)e_1e_3-2e_2+((-2a+2d)/a)e_3,
\end{equation}
where $e_1$, $e_2$ and $e_3$ are the elementary symmetric polynomials in $y, y^q, y^{q^2}$:
\begin{equation}\label{si}\left\{\begin{tabular}{lcl}
$e_1$ & $=$ & $y+y^q+y^{q^2}$ \\
$e_2$ & $=$ & $y^{1+q}+y^{1+q^2}+y^{q+q^2}$\\
$e_3$ & $=$ & $y^{1+q+q^2}$.\\
\end{tabular}\right. 
\end{equation}

The symmetrized third summation polynomial for an elliptic curve in short Weierstrass form is
\begin{equation}\label{g3Wei}G_3(e_1,e_2,e_3)= e_2^2-4e_1e_3-4Be_1-2Ae_2+A^2.\end{equation}
Notice that, while $G_3$ is linear in $e_1$ and $e_3$, $g_3$ is of degree $2$ in each variable. In particular, none of $e_1,e_2,e_3$ is determined uniquely by the other two as is the case of elliptic curves in Weierstrass form. 
However, applying the change of coordinates
\begin{equation}\label{ti}\left\{\begin{tabular}{lcl}
$t_1$ & $=$ & $e_1$ \\
$t_2$ & $=$ & $e_3+e_2$\\
$t_3$ & $=$ & $e_3-e_2$\\
\end{tabular}\right.
\end{equation}
to $g_3$, we obtain the polynomial
\begin{equation}\label{g3tilda}
\tilde{g}_3(t_1,t_2,t_3)=t_1^2+(d/a)(t_2t_3+t_1t_2+t_1t_3)+((d/a)-2)t_2+dt_3-1,
\end{equation}
that is linear in both $t_2$ and $t_3$.

Applying Weil restriction of scalars to the combination of (\ref{si}) and (\ref{ti}) we obtain
\begin{equation}\label{tiWeil}\left\{\begin{tabular}{lcl}
$t_1$ & $=$ & $3y_0$ \\
$t_2$ & $=$ & $y_0^3-3\mu y_0y_1y_2+\mu y_1^3+\mu^2y_2^3+3y_0^2-3\mu y_1y_2$\\
$t_3$ & $=$ & $y_0^3-3\mu y_0y_1y_2+\mu y_1^3+\mu^2y_2^3-3y_0^2+3\mu y_1y_2$\\
\end{tabular}\right. \end{equation}
which express $t_1,t_2,t_3$ as polynomials in $y_0,y_1,y_2$.

\medskip
{\bf Point Compression}. For compression of a point $P=(x,y)\in \mathcal{T}_3$ we use the first two coordinates from~(\ref{ti}) and (\ref{tiWeil}), obtaining
$$\mathcal{R}(P)= (t_1,t_2) = (3y_0,y_0^3-3\mu y_0y_1y_2+\mu y_1^3+\mu^2y_2^3+3y_0^2-3\mu y_1y_2).$$
If we compute $t_2$ as $(y_0+1)(y_0^2-3\mu y_1y_2)+\mu y_1^3+\mu^2y_2^3+2y_0^2$, the cost of computing $\mathcal{R}(P)$ is 3S+4M in $\FF_q$. In the case of elliptic curves in short Weierstrass form, computing the representation of a point is less expensive, as it takes 1S+1M in $\FF_q$ or 1M  in $\FF_q$ with the two methods presented in~\cite[Section 5]{EM1}. 

\medskip
{\bf Point Decompression}. In order to decompress $(t_1,t_2)\in\im\mathcal{R}$ we proceed as follows.\\
\newline
\noindent
{\bf 1.} Given $(t_1,t_2)\in \im\mathcal{R}$, solve $\tilde{g}_3(t_1,t_2,t_3)=0$  for $t_3$. 
If $t_1+t_2+a=0$, then $\tilde{g}_3(t_1,t_2,t_3)=0$ for all $t_3\in\FF_q$. 
If $t_1+t_2+a\neq 0$, then
$$t_3 = -\frac{((d/a)-2)t_2+(d/a)t_1t_2+(t_1+1)(t_1-1)}{(d/a)(t_1+t_2+a)}.$$
Hence $t_3$ can be computed with 3M+1I in $\FF_q$.\\
\newline
{\bf 2.} Given $(t_1,t_2,t_3)$, we solve system (\ref{tiWeil}) for $y_0$, $y_1$, $y_2$.
Notice that, since the $t_i$'s are obtained from the $e_i$'s by a linear change of coordinates, all considerations from \cite{EM1} apply to our situation. In particular, one can compute $y$ from $(t_1,t_2,t_3)$ with at most 3S+3M+1I, 1 square root and 2 cube roots in $\mathbb{F}_q$.\\

Summarizing, the complete decompression algorithm takes at most 3S+6M+2I, 1 square root, and 2 cube roots in $\mathbb{F}_q$. For elliptic curves in short Weierstrass form, decompression takes at most 3S+5M+2I, 1 square root, and 2 cube roots in $\mathbb{F}_q$ or 4S+4M+2I, 1 square roots and 2 cube roots in $\mathbb{F}_q$, depending on the method used. We refer the interested reader to \cite[Section 5]{EM1} for details on the complexity of the computation for curves in short Weierstrass form.

\begin{remark}
Notice that one can also use $(t_1,t_3)$ as an optimal representation of $(x,y)\in \mathcal{T}_3$, and then solve $\tilde{g_3}$ for $t_2$ in order to recover $y$. This choice is analogous to the one we have made, and the computational cost of compression and decompression does not  change.
\end{remark}

\begin{remark}\label{slow}
The symmetry of twisted Edwards curves makes the computation of point addition on these curves more efficient than on elliptic curves in short Weierstrass form. However, the same symmetry results in summation polynomials of higher degree and with a denser support. This explains our empirical observation that the summation polynomials in the elementary symmetric functions for elliptic curves in short Weierstrass form are sparser than those for twisted Edwards curves for $n=3,5$, even though for both curves they have the same degree $2^{n-2}$. For $n=3$, this behavior is apparent if one compares equations (\ref{g3Ed}) and (\ref{g3Wei}).
Therefore, one should expect that compression and decompression for a representation based on summation polynomials  for twisted Edwards curves are less efficient than for elliptic curves in short Weierstrass form. This is confirmed by  our findings.
\end{remark}

The following examples and statistics have been implemented in Magma \cite{magma}.

\begin{example}\label{exsem3}
Let $q=2^{79}-67$ and $\mu=3$.
We choose random curves, defined and birationally equivalent over $\mathbb{F}_q$:
$$E_{a,d}: 31468753957068040687814x^2+y^2=1+192697821276638966498997x^2y^2$$
and
$$E: y^2 = x^3 + 292467848427659499478503x + 361361026736404004345421.$$
We choose  a random point of trace zero $P'\in E(\FF_{q^3})$, and let $P$ be the corresponding point on $E_{a,d}$.
For brevity, here we only write the \textit{x}-coordinates of points of $E$ and the \textit{y}-coordinates of points of $E_{a,d}$:
$$P'= 346560928146076959314753\xi^2 + 456826539628535981034212\xi + 344167470403026652826672,$$
$$P=208520713897518236215966\xi^2 + 451121944550219947368811\xi + 68041089860429901306252.$$
We represent the points of $E$ using the compression coordinates $(t_1,t_2)$ from \cite[Section 5]{EM1}. Denote by $\mathcal{R}$ and $\mathcal{R}'$ the representation maps on $E_{a,d}$ and $E$, respectively.
We compute
$$\mathcal{R}'(P') = (344167470403026652826672, 334324534997495805088214),$$
$$\mathcal{R}(P)= (204123269581289703918756, 98788782936076524413527).$$
We now apply the corresponding decompression algorithms to $\mathcal{R}'(P')$ and $\mathcal{R}(P)$. We obtain
$$\mathcal{R'}^{-1}(344167470403026652826672, 334324534997495805088214)=$$
$$\{346560928146076959314753\xi^2 + 456826539628535981034212\xi + 344167470403026652826672,$$
$$164759498614507503187493\xi^2 + 361520690988197751534381\xi + 344167470403026652826672,$$
$$93142483046730124850775\xi^2 + 390578588997895442137449\xi + 344167470403026652826672\},$$
which are exactly the \textit{x}-coordinate of $P'$ and its Frobenius conjugates.
Similarly $$\mathcal{R}^{-1}(204123269581289703918756, 98788782936076524413527)=$$
$$\{208520713897518236215966\xi^2 + 451121944550219947368811\xi + 68041089860429901306252,$$
$$539321536961066855011167\xi^2 + 237431391097642968386719\xi + 68041089860429901306252,$$
$$461083568756044083478909\xi^2 + 520372483966766258950512\xi + 68041089860429901306252\},$$
which are exactly the \textit{y}-coordinate of $P$ and its Frobenius conjugates.
\end{example}

We now give an estimate of the average time of compression and decompression for groups of different bit-size.
We consider primes $q_1$, $q_2$, and $q_3$ such that $3|q_i-1$ for all $i$, of bit-length $96$, $112$, and $128$,  respectively. For each $q_i$, we consider five pairs of birationally equivalent curves $(E,E_{a,d})$ defined over $\mathbb{F}_{q_i}$, such that the order of $\mathcal{T}_3$ is prime of bit-length respectively $192$, $224$ and $256$.
On each pair of curves we randomly choose $20'000$ pairs of points $(P',P)$ of trace zero, as in Example \ref{exsem3}.
For each pair of points, we compute
$\mathcal{R'}(P'), \mathcal{R}(P), \mathcal{R'}^{-1}(\mathcal{R'}(P')), \mathcal{R}^{-1}(\mathcal{R}(P)).$
For each computation, we consider the average time in milliseconds for each curve, and then the averages over the five curves. The average computation times are reported in the table below.

\bigskip
\noindent\textbf{Table 1}.\\
\newline
\begin{tabular}{|l|c|c|c|}
\hline
Bit-length of $|\mathcal{T}_3|$ & $192$ & $224$ & $256$   \\
\hline
Compression on $E$  & $0.006$ & $0.005$ & $0.006$  \\
\hline
Compression on $E_{a,d}$  & $0.016$ & $0.017$ & $0.015$  \\
\hline
Decompression on $E$  & $0.81$ & $2.40$ & $1.20$  \\
\hline
Decompression on $E_{a,d}$ & $0.88$ & $2.44$ & $1.17$ \\
\hline
\end{tabular}\\

\bigskip
The following table contains the ratios between the average times for point compression and decompression on elliptic curves in short Weierstrass form and twisted Edwards curves.

\bigskip
\noindent\textbf{Table 2}.\\
\newline
\begin{tabular}{|l|c|c|c|}
\hline
 Bit-length of $|\mathcal{T}_3|$& $192$ & $224$ & $256$ \\
\hline
Comp on $E$ / Comp on $E_{a,d}$ & $0.375$ & $0.294$ & $0.400$ \\
\hline
Dec on $E$ / Dec on $E_{a,d}$ & $0.920$ & $0.984$ & $1.026$ \\
\hline
\end{tabular}\\

\subsection{Explicit equations, complexity, and timings for $n=5$}\label{n=5}

In this subsection we treat in detail the case $n=5$. We compute explicit equations for compression and decompression, give an estimate of the complexity of the computations in terms of the number of operations, and give some timings computed in Magma. We also compare with the results obtained in~\cite{EM1} for elliptic curves in short Weierstrass form.

The fifth Semaev polynomial $f_5$ for a twisted Edwards curve has degree $40$, while for curves in short Weierstrass form it has degree $32$. The first polynomial also contains many more terms than the second. This agrees with what we observed in Remark~\ref{slow} for the case $n=3$. The symmetrized fifth summation polynomial $g_5$ has degree $8$ for both Weierstrass and Edwards curves. However, for Edwards curves $g_5$ has degree $8$ in each variable, while for elliptic curves in short Weierstrass form it has degree $6$ in some of the variables. Because of these reasons, we expect that compression and decompression for a trace zero subgroup group coming from a twisted Edwards curve are less efficient than for one coming from a curve in short Weierstrass form. 

For fields such that $16 | q-1$, we perform a linear change of coordinates on the $s_i$'s in order to obtain a polynomial $\tilde{g}_5$, of degree strictly less than $8$ in some variable. 
The polynomial $g_5$ is too big to be printed here. However, denoting by $(g_5)_8$ the part of $g_5$ which is homogeneous of degree $8$, we have:
\begin{equation}\label{g58}(g_5)_8(e_1,\dots,e_5)= e_1^8 + (d/a)^4(e_2^8+e_3^8)+(d/a)^8(e_4^8+e_5^8).\end{equation}
Let $\mu_1\in\overline{\mathbb{F}}_q$ be a primitive $16$-th roots of unity.
Then we can factor $t^8+s^8$ over $\FF_q$ as
$$t^8+s^8=(t-\mu_1s)(t+\mu_1s)r_6(t,s).$$
Therefore, (\ref{g58}) can be written in the form
$$(g_5)_8=e_1^8+(d/a)^4(e_2-\mu_1e_3)(e_2+\mu_1e_3)p_6(e_2,e_3)+(d/a)^8(e_4^8+e_5^8).$$
Hence, after performing the change of coordinates
$$\left\{\begin{tabular}{lcl}
$t_2$ & $=$ & $e_2-\mu_1e_3$ \\
$t_3$ & $=$ & $e_2+\mu_1e_3$ \\
$t_i$ & $=$ & $e_i \mbox{ for } i=1,4,5$ 
\end{tabular}\right.$$
we obtain a polynomial $\tilde{g}_5(t_1,\ldots,t_5)$ of degree $8$ in $t_1,t_4,t_5$, and degree $7$ in $t_2,t_3$.

\begin{example}
Let $q = 2^{10}-3$, $\mu=2$. Consider the Edwards curve $E_{1,486}$ of equation $x^2+y^2 = 1+6x^2y^2$.
Let $P\in \mathcal{T}_5$ be the point
$$P =(u,v)= (951\xi^4 + 338\xi^3 + 246\xi^2 + 934\xi + 133, 650\xi^4 + 927\xi^3 + 301\xi^2 + 171\xi + 973).$$
The compression of $P$ is $\mathcal{R}(P)=(e_1,e_2,e_3,e_4)=(686,289,865,418).$
In order to decompress, we solve
$$g_5(e_1,e_2,e_3,e_4,t) = g_5(686,289,865,418,t)=$$ 
$$71t^8 + 705t^7 + 1007t^6 + 970t^5 + 233t^4 + 1014t^3 + 356t^2 + 198t + 575=0,$$
which has a unique solution $e_5=790\in\mathbb{F}_{q}$.
In order to recover the value of $y$ up to Frobenius conjugates, we find a root in $\FF_{q^5}$ of 
$$y^5 - e_1y^4 + e_2y^3 -e_3y^2 + e_4y -e_5 = y^5 + 335y^4 + 289y^3 + 156y^2 + 418y + 231.$$
Notice that the five roots are Frobenius conjugates of each other. From one $y\in\FF_{q^5}$ we can recompute $x$ via the curve equation, hence recover one of the Frobenius conjugates of $\pm P$.
So the decompression algorithm returns $\mathcal{R}^{-1}(\mathcal{R}(P))=\{\pm P, \pm\varphi(P), \pm\varphi^2(P), \pm\varphi^3(P), \pm\varphi^4(P)\}.$
\end{example}

We now give an example that presents some indeterminacy in the decompression algorithm. 

\begin{example}\label{es_sem_ind}
Let $q = 2^{10}-3$ and consider the Edwards curve
$$E_{210,924} : 210x^2+y^2= 1 + 924x^2y^2$$
and the point
$$P = (1020\xi^4 + 713\xi^3 + 158\xi^2 + 745\xi + 515, 891\xi^4 + 557\xi^3 + 135\xi^2 + 976\xi + 62) \in \mathcal{T}_5.$$
The compressed representation of $P$ is $\mathcal{R}(P) = (e_1,e_2,e_3,e_4) = (310,887,19,660).$
The decompressing equation is
$$g_5(e_1,e_2,e_3,e_4,t) = 62t^8 +502t^7 + 388t^6 + 294t^5 + 2t^4 + 466t^3 + 723t^2 + 55t + 388= 0,$$
which has solutions $e_5= 428, e'_5= 835, e''_5=550 \in \mathbb{F}_q$. By solving the equation
$$y^5 - e_1y^4 + e_2y^3 -e_3y^2 + e_4y -e_5=y^5 + 310y^4 + 887y^3 + 19y^2 + 660y + 593=0$$
we recover the $y$-coordinate of $P$ and all its Frobenius conjugates.
By solving the equation
$$y^5 - e_1y^4 + e_2y^3 -e_3y^2 + e_4y -e'_5=y^5 + 310y^4 + 887y^3 + 19y^2 + 660y + 186=0$$
we find roots in $\FF_{q^5}$, which do not correspond to points of trace zero.
By solving the equation
$$y^5 - e_1y^4 + e_2y^3 -e_3y^2 + e_4y -e''_5=y^5 + 310y^4 + 887y^3 + 19y^2 + 660y + 471=0$$
we find $Q\in\mathcal{T}_5$ which is not a Frobenius conjugate of $P$.
Hence in this case $$\mathcal{R}^{-1}(\mathcal{R}(P))=\{\pm P,\ldots,\pm\varphi^4(P),\pm Q,\ldots,\pm\varphi^4(Q)\}.$$
\end{example}

Denote by $\mathcal{T}_5/\sim$ the quotient of $\mathcal{T}_5$ by the equivalence relation that identifies opposite points and Frobenius conjugates. The representation (\ref{repr_sem}) induces a representation 
$$\mathcal{R}':\mathcal{T}_5/\sim \ \longrightarrow \ \FF_q^4.$$
In the previous example we show that $\mathcal{R}'$ is not injective. Nevertheless, an easy heuristic argument shows that a generic $(e_1,\ldots,e_4)\in\im\mathcal{R}'$ has exactly one inverse image.
In order to support the heuristics, we tested $15'000$ random points in the trace zero subgroup $\mathcal{T}_5$ of $15$  Edwards curves. The groups had prime cardinality and bit-length $192, 224,$ and $256$. For any random point $P$ we computed the cardinality of $\mathcal{R'}^{-1}(\mathcal{R'}(P))$, and found that it is $1$ for about $91\%$ of the points, $2$ for about $8.5\%$ of the points, and $3$ for about $0.5\%$ of the points. We also found a few points for which $|\mathcal{R'}^{-1}(\mathcal{R'}(P))|=4$, but the percentage was less than $0.02\%$. Finally, we did not find any points for which $4<|\mathcal{R'}^{-1}(\mathcal{R'}(P))|\leq 8$.

In order to test the efficiency of the compression and decompression algorithms for $n=5$, we have implemented them in Magma \cite{magma}.
We consider primes $q_1$, $q_2$, and $q_3$ of bit-length $48$, $56$, and $64$, respectively. We choose primes such that $5|q_i-1$ for all $i$.
For each $q_i$ we consider five pairs of birationally equivalent curves $(E,E_{a,d})$ defined over $\mathbb{F}_{q_i}$, such that the order of $\mathcal{T}_5$ is prime of bit-length $192$, $224$, and $256$,  respectively.
The following table contains the average times for compression and decompression in milliseconds. Each average is computed on a set of 20'000 randomly chosen points on each of the five curves.

\bigskip
\noindent\textbf{Table 3}.\\
\newline
\begin{tabular}{|l|c|c|c|c|}
\hline
Bit-length of $|\mathcal{T}_5|$ & $192$ & $224$ & $256$   \\
\hline
Compression on $E$  & $0.057$ & $0.055$ & $0.060$  \\
\hline
Compression on $E_{a,d}$  & $0.049$ & $0.058$ & $0.053$ \\
\hline
Decompression on $E$  & $64.17$ & $104.31$ & $121.51$  \\
\hline
Decompression on $E_{a,d}$ & $63.66$ & $104.45$ & $121.42$\\
\hline
\end{tabular}\\

\bigskip
The following table contains the ratios between the average times for point compression and decompression on elliptic curves in short Weierstrass form and twisted Edwards curves. 

\bigskip
\noindent\textbf{Table 4}.\\
\newline
\begin{tabular}{|l|c|c|c|c|}
\hline
Bit-length of $|\mathcal{T}_5|$ & $192$ & $224$ & $256$ \\
\hline
Comp on $E$ / Comp on $E_{a,d}$ & $1.163$ & $0.948$ & $1.132$  \\
\hline
Dec on $E$ / Dec on $E_{a,d}$ & $1.008$ & $0.999$ & $1.001$ \\
\hline
\end{tabular}\\

\section{An optimal representation using rational functions}

Let $E_{a,d}$ be a twisted Edwards curve defined over $\FF_q$. 
In this section, we propose another optimal representation for the trace zero subgroup $\mathcal{T}_n\subset E_{a,d}(\FF_{q^n})$ using rational functions.

In~\cite{EM2} the authors propose to represent an element $P\in \mathcal{T}_n$ via the coefficients of the rational function which corresponds to the principal divisor $P+\varphi(P)+\ldots+\varphi^{n-1}(P)-n\OO$ on the elliptic curve. Optimality of the representation depends on the fact that the rational function associated to this divisor has a special form, and can therefore be represented using $n-1$ coefficients in $\FF_q$. If we consider a principal divisor of the form $P+\varphi(P)+\ldots+\varphi^{n-1}(P)-n\OO$ on the twisted Edwards curve $E_{a,d}$, there are several questions that need to be answered. E.g., the rational function associated to this divisor is not a polynomial in general, so one needs to overcome some difficulties in order to successfully carry out the same strategy. 

We start with some preliminaries results on rational functions on a twisted Edwards curve. If $h$ is a rational function on $E_{a,d}$, we denote by $\divv(h)$ the divisor of the homogeneous rational function associated to $h$ on the projective closure of $E_{a,d}$. Throughout the section we use $(u,v)$ for the coordinates of the point and $x,y$ for the variables of the rational functions, in order to avoid confusion.

\begin{lemma}\label{lemma1_ratfun}
Let  $c\in k$ such that $ad^{-1}=c^2$, where $k=\FF_q$ or $k=\FF_{q^2}$ depending on whether $ad^{-1}$ is a quadratic residue in $\FF_q$ or not. 
Let $R(x,y) \in k(x,y)$ be a rational function over $E_{a,d}$. Then $R$ can be written in the form
$$R(x,y) = (y-c)^{k_1}(y+c)^{k_2}\frac{r_1(y)+xr_2(y)}{r_3(y)},$$
modulo $E_{a,d}$, where $r_1, r_2, r_3\in k[y]$, $\gcd\{r_1,r_2,r_3\}=1$, $r_3(\pm c)\not = 0$, and $k_1$, $k_2 \leq 0$.
\end{lemma}

\begin{proof}
Using the relation $x^2= \frac{(1-y^2)}{(a-dy^2)}$, we can write $R(x,y)$ in the form
$$R(x,y)=\frac{s_1(y)+xs_2(y)}{s_3(y)+xs_4(y)},$$
where $s_i(y) \in k[y]$ for $1\leq i \leq 4$.
Multiplying and dividing by $s_3(y)-xs_4(y)$, we obtain:
$$R(x,y)=\frac{t_1(y)+xt_2(y)}{t_3(y)},$$
where $t_i(y) \in k[y]$ for $1\leq i \leq 3$.
Simplifying the fraction and factoring $y-c$ and $y+c$ as much as possible from the denominator, we obtain the thesis.\end{proof}
 
\begin{lemma}\label{ratfunpoli}
In the setting of Lemma \ref{lemma1_ratfun}, assume that $R$ has poles at most at the points at infinity $\Omega_1$ and $\Omega_2$. Then
$$R(x,y) = (y-c)^{k_1} (y+c)^{k_2} (q_1(y)+xq_2(y)),$$
modulo $E_{a,d}$, where $q_1(y)$, $q_2(y) \in k[y]$, $q_i(\pm c)\neq 0$ for $i=1,2$, and $k_1,k_2 \leq 0$.
\end{lemma}

\begin{proof}
By Lemma \ref{lemma1_ratfun} we can write $$R(x,y) =(y-c)^{k_1}(y+c)^{k_2}\frac{r_1(y)+xr_2(y)}{r_3(y)}.$$ Since $(y-c)^{k_1}=0$ and $(y+c)^{k_2}=0$ have no affine zeroes on $E_{a,d}$, $R$ has poles at most at the points at infinity if and only if the order of vanishing of $r_3$ on $E_{a,d}$ at each affine point is less than or equal to the order of vanishing of $r_1+xr_2$ on $E_{a,d}$ at the same point. 

Let $P=(u,v)$ be a point such that $r_3(v)=0$. Write $r_3$ in the form $r_3(y) = (y-v)^mt_3(y),$ where $t_3(v)\neq 0$ and $m > 0$. 
The order of vanishing of $r_3$ on $E_{a,d}$ at $P$ is $m$ if $u \neq 0$, and $2m$ if $u=0$. 
In fact, the only points in which $E_{a,d}$ has a horizontal tangent line are $\OO$ and $\OO'$.
The same holds for the order of vanishing of $r_3$ at $-P$.
From $r_1(v)+ur_2(v)=r_1(v)-ur_2(v)=0$ we obtain that $r_1(v)=ur_2(v)=0$. Therefore, since  $\gcd\{r_1,r_2,r_3\}=1$, we have $r_2(v)\neq 0$ and $u=0$.  The order of vanishing of $r_1+xr_2$ on $E_{a,d}$ at $P$ is $1$, since $P$ is a smooth point and the tangent line at $P$ to the curve of equation $r_1(y)+xr_2(y)$ is not horizontal. But the order of vanishing of $r_3$ on $E_{a,d}$ at $P$ is bigger than $m$, which yields a contradiction.
\end{proof}

In the introduction of this section, we hinted at the difficulty that if $P\in \mathcal{T}_n$ is a point of trace zero on a twisted Edwards curve $E_{a,d}$, the rational function associated to the principal divisor $P+\varphi(P)+\ldots+\varphi^{n-1}(P)-n\OO$ is not in general a polynomial. Lemma~\ref{ratfunpoli} offers a solution to this problem: considering a modified principal divisor, whose associated rational function is a polynomial. 
 
\begin{theorem}\label{Th_es_poli}
Let $E_{a,d}$ be a twisted Edwards curve defined over $\FF_q$ and let $P\in \mathcal{T}_n\subset E_{a,d}(\FF_{q^n})$. Then there exists a polynomial $q_P(x,y)=q_1(y)+xq_2(y)\in\FF_q[x,y]$, with $q_1(y),q_2(y) \in \mathbb{F}_q[y]$, such that
\begin{enumerate}
\item $div(q_P)=P+\varphi(P)+\ldots+\varphi^{n-1}(P)+\OO'-2\Omega_1-(n-1)\Omega_2.$
\item $\max\{\deg(q_1),\deg(q_2)\} =\frac{n-1}{2}.$
\item $q_1(y) = (1+y)\hat{q_1}(y),$ where $\hat{q_1}\in\mathbb{F}_q[y]$ and $\deg(\hat{q_1})\leq\frac{n-3}{2}$.
\item$q_2$ is not the zero polynomial.
\end{enumerate}
\end{theorem}

\begin{proof} 
$1.$ The point $P=(u,v)$ has trace zero, hence $P+\varphi(P)+\ldots+\varphi^{n-1}(P)=\mathcal{O}.$
Then there exists a rational function $f$ on $E_{a,d}$ defined over $\FF_q$ such that
$$\divv(f)=P+\varphi(P)+\ldots+\varphi^{n-1}(P)-n\mathcal{O}.$$
The polynomial $H(x,y) = x(1-y)^{\frac{n-1}{2}} \in \mathbb{F}_q[x,y]$ corresponds to the divisor
$$\divv(H)=n\OO+\OO'-2\Omega_1-(n-1)\Omega_2.$$
Therefore
$$\divv(fH)=P+\varphi(P)+\ldots+\varphi^{n-1}(P)+\OO'-2\Omega_1-(n-1)\Omega_2.$$
By Lemma \ref{ratfunpoli}, we can write
$$fH=(y-c)^{k_1}(y+c)^{k_2}(q_1(y)+xq_2(y)),$$
where $q_1(y), q_2(y)$ are polynomials, $k_1, k_2 \leq 0$, and $ad^{-1}=c^2$.
We now prove that $k_1=k_2=0$ i.e. $fH=q_P$, from which we get part 1. 
For each $1\leq i \leq n$, let $P_i=\varphi^{i-1}(P)$. 
For each $1\leq i \leq n-2$, let $\phi_i$ be the conic with
$$\divv(\phi_i) = (P_1+\ldots+ P_{i-1}+ P_{i})+P_{i+1}+(-(P_1+\ldots+ P_i+ P_{i+1}))+\mathcal{O}'-2\Omega_1-2\Omega_2.$$ Notice that $\phi_i$ exists by~\cite[Theorem 1 and Theorem 2]{GeomInt}, and it is unique up to multiplication by a constant.
For each $1\leq i \leq n-3$, let $h_i$ be the horizontal line through the point $P_1+\dots+P_{i+1}\in E_{a,d}$.
Then
$$\divv(h_i) =(P_1+\ldots+P_{i+1})+(-(P_1+\ldots+P_{i+1}))-2\Omega_2.$$
Since
$\divv(x) = \mathcal{O}+\mathcal{O}'-2\Omega_1$, $\divv(1-y) = 2\mathcal{O}-2\Omega_2$, and $f$ has no zeroes or poles at infinity, 
we have the equality of rational functions: 
$$f = \frac{\phi_1\phi_2\cdots\phi_{n-2}}{x^{n-2}(1-y) h_1 h_2\cdots h_{n-3}},$$
up to multiplication by a nonzero constant.
Therefore
\begin{equation}\label{qP}
fH = \frac{\phi_1\phi_2\cdots\phi_{n-2}(1-y)^{\frac{n-3}{2}}}{h_1 h_2\cdots h_{n-3} x^{n-3}} =\frac{(a-dy^2)^{\frac{n-3}{2}}}{h(y)(1+y)^{\frac{n-3}{2}} } \prod_{i=1}^{n-2}{\phi_i}
\end{equation}
modulo the curve equation, where $h(y)=\prod_{i=1}^{n-3}{h_i}$ and $\deg(h)=n-3$.
For each $1\leq i \leq n-2$, $\phi_i$ is of the form
$\phi_i=B_i(y) x+A_i(y),$
where $B_i(y)$ and $A_i(y)$ are polynomials in $y$ of degree at most $1$, by~\cite[Theorem~1]{GeomInt}. Hence
$$\prod_{i=1}^{n-2}{\phi_i}=H_{n-2}(y) x^{n-2}+H_{n-3}(y) x^{n-3} + \ldots  +H_1(y) x + H_0(y),$$
where each $H_i(y)$ is a polynomial in $y$ of degree at most $n-2$. Then, reducing modulo $E_{a,d}$ we obtain
$$(a-dy^2)^{\frac{n-3}{2}} \prod_{i=1}^{n-2}{\phi_i}(x,y)=R_1(y)+xR_2(y),$$
where each $R_i(y)$ is a polynomial of $\deg(R_i)\leq \max\{\deg(H_j)\} + n-3\leq 2n-5.$
The denominator of (\ref{qP}) divides both $R_1(y)$ and $R_2(y)$ by Lemma~\ref{ratfunpoli}.
Hence, letting
$R_i(y)=q_i(y)h(y)(1+y)^{\frac{n-3}{2}}$
for $i =1,2$, we have that $fH=q_P$. 

\medskip\noindent
$2.$ Using the notation of part $1$, we have
\begin{equation}\label{dis1}\deg(q_i) = \deg(R_i) - \deg(1+y)^{\frac{n-3}{2}}-\deg(h) \leq 2n-5 - \frac{(n-3)}{2}-(n-3)=\frac{n-1}{2}\end{equation}
for $i=1,2$. Moreover, by part $1$
$$\divv(q_{-P})=(-P)+\ldots+\varphi^{n-1}(-P)+\OO'-2\Omega_1-(n-1)\Omega_2,$$
and modulo $E_{a,d}$ $$q_P(x,y)q_{-P}(x,y)=q_1^2(y)-\frac{1-y^2}{a-dy^2}\, q_2^2(y).$$
Since $\divv(a-dy^2)=4\Omega_1-4\Omega_2$, the polynomial
$R_P(y)=(a-dy^2) q_1^2(y) - (1-y^2) q_2^2(y)$ has
$$div(R_P)= (\pm P) + (\pm \varphi(P)) + \dots + (\pm \varphi^{n-1}(P)) + 2\mathcal{O}'-2(n+1)\Omega_2.$$
Hence $(1+y)\prod_{i=0}^{n-1}v^{q^i}|R_P(y)$, therefore
\begin{equation}\label{dis2}n+1 \leq \deg(R_P(y)) \leq 2+2\max\{\deg(q_1),\deg(q_2)\}\end{equation}
and part $2$ follows directly from (\ref{dis1}) and (\ref{dis2}). We have also obtained that $R_P$ is a polynomial of degree exactly $n+1$ with coefficients in $\mathbb{F}_q$ and roots $-1$, $v^{q^i}$, for $0\leq i \leq n-1$: we will need this result in the sequel. 

\medskip\noindent
$3.$ Since $q_P$ vanishes at $\mathcal{O}'=(0,-1)$, then $q_1$ is of the form
$$q_1(y) = (1+y)\hat{q_1}(y),$$ where $\hat{q_1}\in\mathbb{F}_q[y]$ and $\deg(\hat{q_1})\leq\frac{n-3}{2}$.

\medskip\noindent
$4.$ If $q_2$ was the zero polynomial, then $q_P=q_1(y)$ would vanish on $\mathcal{O}'$ with multiplicity at least $2$, contradicting part $1$.

\end{proof}

\bigskip\noindent{\bf Computation of $q_P$.} 
In the proof of the previous theorem we have seen that one can compute the polynomial $q_P$ as
\begin{equation}\label{qPMS}q_P = \frac{\phi_1\phi_2\cdots\phi_{n-2}(1-y)^{\frac{n-3}{2}}}{h_1 h_2\cdots h_{n-3} x^{n-3}},\end{equation}
where 
for each $1\leq i \leq n$, $P_i=\varphi^{i-1}(P)$, 
for each $1\leq i \leq n-2$, $\phi_i$ is the conic through $(P_1+\ldots+ P_{i-1}+ P_{i})$, $P_{i+1}$, $\mathcal{O}'$, $2\Omega_1$ and $2\Omega_2$, 
for each $1\leq i \leq n-3$, $h_i$ is the horizontal line through $P_1+\dots+P_{i+1}\in E_{a,d}$.
Notice that we can easily calculate $\phi_i$ for each $i$, employing the formulas given in ~\cite[Theorem 1 and Theorem 2]{GeomInt}.

We now discuss how to use the polynomial $q_P$ to represent $P$ via $(n-1)$ elements of $\mathbb{F}_q$ plus a bit.
As a consequence of Theorem \ref{Th_es_poli}, $q_P$ has the form
$$q_P(x,y)= (1+y)\left(a_{\frac{n-3}{2}}y^{\frac{n-1}{2}}+\dots + a_1y + a_0\right)+x\left(b_{\frac{n-1}{2}}
y^{\frac{n-1}{2}}+\dots + b_1y + b_0\right),$$
where $a_i, b_j \in \mathbb{F}_q$ for all $i, j$, and $b_{\frac{n-1}{2}}\in\{0,1\}$.
We have therefore obtained an optimal representation for the elements of $\mathcal{T}_n$:
\begin{equation}\label{CompressP}
\begin{array}{rccl}
\mathcal{R}: & \mathcal{T}_n & \longrightarrow & \FF_q^{n-1}\times\FF_2 \\
& P & \longmapsto & \left(a_0,\dots,a_{\frac{n-3}{2}},b_0,\dots,b_{\frac{n-1}{2}}\right).
\end{array}
\end{equation}

We now give the complete algorithm for point compression.\\

\hrule
\begin{algorithm}[{\bf Compression}]\textbf{ }\\
\hrule
\textbf{ }\\
\newline
\textbf{Input} : $P \in \mathcal{T}_n$\\
\textbf{Output} : $\mathcal{R}(P) \in \mathbb{F}_q^{n-1}\times\mathbb{F}_2$\\
\newline
\begin{footnotesize}1: \end{footnotesize} Compute $q_P(x,y)=q_1(y)+xq_2(y)$ using (\ref{qP}) and reducing modulo $E_{a,d}$.\\
\begin{footnotesize}2: \end{footnotesize} Compute $\hat{q_1}(y) = q_1(y)/(1+y)=a_{\frac{n-3}{2}}y^{\frac{n-1}{2}}+\dots + a_1y+a_0$.\\ 
\begin{footnotesize}3: \end{footnotesize} $q_2(y)=b_{\frac{n-1}{2}}y^{\frac{n-1}{2}}+\dots+b_1y+b_0$.\\
\begin{footnotesize}4: \end{footnotesize} $\mathcal{R}(P)\leftarrow (a_0,\dots,a_{\frac{n-3}{2}},b_0,\dots,b_{\frac{n-1}{2}})$.\\
\begin{footnotesize}5: \end{footnotesize} {\bf return} $\mathcal{R}(P)$.\\
\hrule
\end{algorithm}
Correctness of the compression algorithm is a direct consequence of  our previous results.\\
 
Given an $n$-tuple $(\alpha_1,\dots,\alpha_{n-1}, b) \in \mathbb{F}_q^{n-1}\times\mathbb{F}_2$ such that $(\alpha_1,\dots,\alpha_{n-1}, b)=\mathcal{R}(P)$ for some $P\in \mathcal{T}_n$, we want to compute the decompression $\mathcal{R}^{-1}(\alpha_1,\dots,\alpha_{n-1}, b)$. We start with some preliminary results. The next lemma guarantees that the $x$-coordinate of $P$ can be computed from its \textit{y}-coordinate and the polynomial $q_P$. 
 
\begin{lemma}\label{xcoord}
Let $P=(u,v)\in \mathcal{T}_n$, let $q_P(x,y)=q_1(y)+xq_2(y)\in\FF_q[x,y]$ be the polynomial
with $div(q_P)=P+\varphi(P)+\ldots+\varphi^{n-1}(P)+\OO'-2\Omega_1-(n-1)\Omega_2.$ Then:
$q_2(v)=0$ if and only if  $P=\mathcal{O}$.
\end{lemma}

\begin{proof}
If $q_2(v)=0$, then $q_1(v)=0$, hence $q_P(-u,v)=0$.
Since the affine points of the curve on which $q_P$ vanishes are exactly $\mathcal{O}'$ and $\varphi^i(P)$ for $0\leq i  \leq n-1$ by Theorem \ref{Th_es_poli} and $\mathcal{O}' \not \in \mathcal{T}_n$, then $-P=\varphi^i(P)$ for some $i$. 
If $i=0$, we have $-P=P$, hence $P=\mathcal{O}$. If $i\not=0$, then $(-u,v)=(u^{q^i},v^{q^i})$ for some $i \in \{1,\dots,n-1\}$. Then $v \in \mathbb{F}_{q^i}\cap\mathbb{F}_{q^n}=\mathbb{F}_q$ and $u^{q^{2i}}=u \in \mathbb{F}_{q^{2i}}\cap\mathbb{F}_{q^n}=\mathbb{F}_q$. Hence $P \in E_{a,d}(\mathbb{F}_q)$ and $-P=\varphi^i(P)=P$, from which $P=\mathcal{O}$.

Conversely, if $P=\mathcal{O}$ then $q_P(x,y)=x(1-y)^{\frac{n-1}{2}}$ and $q_2(1)=0$.
\end{proof}

Given $q_P(x,y)$, we can compute a polynomial $Q_P(y)$ whose roots are exactly the Frobenius conjugates of the $y$-coordinate of $P$. This will be used in our decompression algorithm.

\begin{proposition}\label{prop_QP}
Let $P=(u,v)\in \mathcal{T}_n$, let $q_P(x,y)=(1+y)\hat{q_1}(y)+xq_2(y)\in\FF_q[x,y]$ be the polynomial
with $div(q_P)=P+\varphi(P)+\ldots+\varphi^{n-1}(P)+\OO'-2\Omega_1-(n-1)\Omega_2.$
Define $$Q_P(y)=(a-dy^2)(1+y)\hat{q_1}^2(y) +(y-1)q_2^2(y).$$
Then $Q_P(y)\in\FF_q[y]$, $\deg Q_P=n$, and its roots are $v,v^q,\ldots,v^{q^{n-1}}$.
\end{proposition}

\begin{proof}
Let $R_P=(a-dy^2) q_1^2(y) - (1-y^2) q_2^2(y)=(1+y)[(a-dy^2)\hat{q_1}(y) - (1-y) q_2^2(y)]$.
Then $Q_P(y) = (1+y)^{-1}\cdot R_P(y)$, and the claim follows by Theorem~\ref{Th_es_poli}. 
\end{proof}
 
We are now ready to give the decompression algorithm.
\newline
\hrule
\begin{algorithm}[{\bf Decompression}]\textbf{ }\\
\hrule
\textbf{ }\\
\newline
\textbf{Input} : $(\alpha_1,\dots,\alpha_{n-1},b) \in \mathbb{F}_q^{n-1}\times\mathbb{F}_2$\\
\textbf{Output} : $P=(u,v) \in \mathcal{T}_n$ with $\mathcal{R}(P)=(\alpha_1,\dots,\alpha_{n-1},b)$\\
\newline
\begin{footnotesize}1: \end{footnotesize} $\hat{q}_1(y) \leftarrow \alpha_{\frac{n-1}{2}}y^{\frac{n-3}{2}}+\dots+\alpha_2y+\alpha_1.$\\
\begin{footnotesize}2: \end{footnotesize} $q_2(y) \leftarrow by^{\frac{n-1}{2}}+\alpha_{n-1}y^{\frac{n-3}{2}}+\dots+\alpha_{\frac{n+3}{2}}y+\alpha_{\frac{n+1}{2}}.$\\
\begin{footnotesize}3: \end{footnotesize} $Q_P(y) \leftarrow (a-dy^2)\cdot(1+y)\cdot \hat{q}_1^2(y)+(y-1)\cdot q_2^2(y)$.\\
\begin{footnotesize}4: \end{footnotesize} $v \leftarrow$ one root of $Q_P(y)$.\\
\begin{footnotesize}5: \end{footnotesize} \textbf{if } $v=1$ \textbf{ then } $u \leftarrow 0$ \textbf{ else } $u \leftarrow -\frac{\hat{q}_1(v)(v+1)}{q_2(v)}$ \textbf{ endif  }\\
\begin{footnotesize}6: \end{footnotesize} {\bf return} $(u,v)$. \\
\hrule
\end{algorithm}

\begin{remark}
Let $P\in\mathcal{T}_n$ be a point with $\mathcal{R}(P)=(\alpha_1,\dots,\alpha_{n-1},b)$. By Theorem~\ref{Th_es_poli} the Frobenius conjugates of $P$ are the only other points of $\mathcal{T}_n$ with the same representation. Correctness of the first four lines of the algorithm follows from Proposition \ref{prop_QP} and correctness of line $5$ follows from Lemma \ref{xcoord}. Hence the given algorithm correctly recovers the point $P$, up to Frobenius conjugates.
\end{remark}


\subsection{Explicit equations, complexity, and timings for $n=3$}

In this subsection we give explicit equations and perform some computations for $n=3$. We estimate the number of operations needed for the compression and decompression, and present some timings obtained with Magma. We also make comparisons with trace zero subgroups of elliptic curves in short Weierstrass form treated in \cite{EM2}. \\

{\bf Point Compression}. Let $P=(u,v) \in T_3$. By Theorem~\ref{Th_es_poli}, we may write
$$q_P(x,y)=\hat{q_1}(y)(1+y)+xq_2(y)=a_0(1+y)+x(b_1y+b_0),$$
where $a_0 \mbox{, } b_0 \in \mathbb{F}_q \mbox{, } b_1 \in \{0,1\}$.

If $P\not \in E_{a,d}(\mathbb{F}_q)$, let $t=\frac{v+1}{u}$. Notice that $u\not=0$, since $u=0$ implies $P=\mathcal{O}$, hence $P\in E_{a,d}(\mathbb{F}_q)$. \\
\newline
\medskip\noindent
{\bf 1.}  If $t^q-t\not=0$, by Theorem $1$ of \cite{GeomInt}
$$\mathcal{R}(P)=(a_0,b_0,b_1)=\left(-\frac{v^q-v}{t^q-t},-a_0t-v,1\right).$$
Computing $t$ from $u$ and $v$ takes $1$M+$1$I in $\mathbb{F}_{q^3}$. Once we have $t$, the situation is analogous to the case of elliptic curves in short Weierstrass form. Hence we refer to \cite[Section 5.1]{EM2} for a detailed discussion of how to efficiently compute $\mathcal{R}(P)$. In particular, it is shown that one can compute $a_0$ and $b_0$ with $2$S+$6$M +$1$I in $\mathbb{F}_q$.
Summarizing, point compression in this case takes $1$M+$1$I in $\mathbb{F}_{q^3}$ and $2$S+$6$M +$1$I in $\mathbb{F}_q$. Due to the calculation of $t$, it is more expensive than that for elliptic curves in short Weierstrass form. \\
\newline
\medskip\noindent
{\bf 2.} If $t^q-t=0$, then $q_P$ is the line passing through $P$ and $\mathcal{O}'$ by \cite[Theorem 1]{GeomInt}. Hence
\begin{equation}\label{ComprP3_1b}
\mathcal{R}(P)=\left(-t^{-1},1,0\right).
\end{equation}
Since $\mathcal{O}'\not \in {T}_3$, then $t\not=0$.
In this case point compression requires only $1$M + $1$I in $\mathbb{F}_{q^3}$.\\

If $P \in E_{a,d}(\mathbb{F}_q)$, then the computation takes place in $\FF_q$ instead of $\FF_{q^3}$, hence we expect the complexity to be lower. We carry on a precise operation count, as in the previous case.\\
\newline
\medskip\noindent
{\bf 3.} If $du^2v-1\not=0$, by \cite[Theorem 1]{GeomInt}
$$\mathcal{R}(P)=\left(\frac{u(1-v)}{du^2v-1},\frac{v-au^2}{du^2v-1},1\right).$$
Therefore, point compression takes $1$S+$4$M+$1$I in $\mathbb{F}_q$.\\
\newline
\medskip\noindent
{\bf 4.} If $du^2v-1=0$, then the situation is analogous to {\bf 2.} and $\mathcal{R}(P)$ is given by (\ref{ComprP3_1b}). Hence point compression requires $1$M + $1$I in $\mathbb{F}_{q}$.
\newline\medskip

Since {\bf 1.} is the generic case, the expected complexity of point compression is $1$M+$1$I in $\mathbb{F}_{q^3}$ and $2$S+$6$M +$1$I in $\mathbb{F}_q$.\\

{\bf Point Decompression}. Let $(\alpha_1,\alpha_2,b) \in \mathbb{F}_q^2\times \mathbb{F}_2$ and $P=(u,v)\in\mathcal{T}_3$ such that $\mathcal{R}(P)=(\alpha_1,\alpha_2,b)$. In order to recover $P$ from $\mathcal{R}(P)$, we want to find the roots of
$$
Q_P(y)= (b-d\alpha_1^2)y^3+(-d\alpha_1^2+2\alpha_2b-b)y^2+(a\alpha_1^2-2\alpha_2b+\alpha_2^2)y+(a\alpha_1^2-\alpha_2^2).
$$
They are the solutions system
\begin{equation}\label{systemn3}
\left\{\begin{tabular}{lcl}
$y+y^q+y^{q^2}$ & $=$ & $c(d\alpha_1^2-2\alpha_2b+b)$ \\
$y^{q+1}+y^{q^2+1}+y^{q^2+q}$ & $=$ & $c(a\alpha_1^2-2\alpha_2b+\alpha_2^2)$\\
$y^{1+q+q^2}$ & $=$ & $c(-a\alpha_1^2+\alpha_2^2)$\\
\end{tabular}\right.\end{equation}
where $c=(b-d\alpha_1^2)^{-1}$. Notice that $(b-d\alpha_1^2)\not=0$, since $Q_P$ has degree $3$  by Proposition \ref{prop_QP}.\\

Computing the constant terms of (\ref{systemn3}) takes $2$S+$3$M+1$I$ in $\mathbb{F}_q$.
Computing a solution of the system takes at most $3$S+$3$M+$1$I, one square root and two cube roots in $\mathbb{F}_q$, as shown in \cite{EM2}. Finally, computing $u$ from $v$ requires $2$M+$1$I in $\mathbb{F}_{q^3}$.
Summarizing, for $n=3$ point decompression takes at most $2$M+$1$I in $\mathbb{F}_{q^3}$ and $5$S+$6$M+$2$I, one square root and two cube roots in $\mathbb{F}_q$. It is more expensive than that for elliptic curves in short Weierstrass form, which takes at most $1$M in $\mathbb{F}_{q^3}$ and $5$S+$4$M+1$I$, one square root and two cube roots in $\mathbb{F}_q$.\\

We now give an example and some statistics implemented in Magma. We follow the same setup as in Example~\ref{exsem3}, and compare with the method for elliptic curves in short Weierstrass form proposed in \cite{EM2}. 

\begin{example}
Let $q=2^{79}-67$ and $\mu=3$.
We choose random, birationally equivalent curves defined over $\FF_q$:
$$E_{a,d}: 31468753957068040687814x^2+y^2=1+192697821276638966498997x^2y^2$$
and
$$E: y^2 = x^3 + 292467848427659499478503x + 361361026736404004345421.$$
We choose a random point $P'\in E(\FF_{q^3})$ of trace zero, and let $P$ be the corresponding point on $E_{a,d}$.
For brevity, we only write the \textit{x}-coordinates of points of $E$ and the \textit{y}-coordinates of points of $E_{a,d}$:
$$P'= 346560928146076959314753\xi^2 + 456826539628535981034212\xi + 344167470403026652826672,$$
$$P=208520713897518236215966\xi^2 + 451121944550219947368811\xi + 68041089860429901306252.$$
We denote by $\mathcal{R}$ and $\mathcal{R}'$ the representation maps on $E_{a,d}$ and $E$, respectively. We compute:
$$\mathcal{R}'(P') = (\gamma_0,\gamma_1)=(48823870679406912678832, 283451751560764957720302),$$
$$\mathcal{R}(P)= (a_1,b_0,b_1)=(313084342552232820027816, 535814703179324297074161,1).$$
Applying the decompression algorithms to $\mathcal{R}'(P')$ and $\mathcal{R}(P)$, we obtain
$$\mathcal{R'}^{-1}(48823870679406912678832, 283451751560764957720302)=$$
$$\{346560928146076959314753\xi^2 + 456826539628535981034212\xi + 344167470403026652826672,$$
$$164759498614507503187493\xi^2 + 361520690988197751534381\xi + 344167470403026652826672,$$
$$93142483046730124850775\xi^2 + 390578588997895442137449\xi + 344167470403026652826672\},$$
which are the \textit{x}-coordinates of $P'$ and its Frobenius conjugates.
Similarly $$\mathcal{R}^{-1}(313084342552232820027816, 535814703179324297074161,1)=$$
$$\{208520713897518236215966\xi^2 + 451121944550219947368811\xi + 68041089860429901306252,$$
$$539321536961066855011167\xi^2 + 237431391097642968386719\xi + 68041089860429901306252,$$
$$461083568756044083478909\xi^2 + 520372483966766258950512\xi + 68041089860429901306252\},$$
which are the \textit{y}-coordinates of $P$ and its Frobenius conjugates.
\end{example}

We now give an estimate of the average time of compression and decompression for groups of different bit-size.
We consider primes $q_1$, $q_2$, and $q_3$ such that $3|q_i-1$ for all $i$, of bit-length $96$, $112$, and $128$,  respectively. For each $q_i$, we consider five pairs of birationally equivalent curves $(E,E_{a,d})$ defined over $\mathbb{F}_{q_i}$, such that the order of $\mathcal{T}_3$ is prime of bit-length respectively $192$, $224$ and $256$.
On each pair of curves we randomly choose $20'000$ pairs of points $(P',P)$ of trace zero which correspond to each other via the birational isomorphism between the curves. For each pair of points, we compute
$\mathcal{R'}(P'), \mathcal{R}(P), \mathcal{R'}^{-1}(\mathcal{R'}(P')), \mathcal{R}^{-1}(\mathcal{R}(P)).$
For each computation, we consider the average time in milliseconds for each curve, and then the averages over the five curves. The average computation times are reported in the table below.\\
\newline
\textbf{Table 5}.\\
\newline
\begin{tabular}{|l|c|c|c|}
\hline
Bit-length of $|\mathcal{T}_3|$ & $192$ & $224$ & $256$   \\
\hline
Compression on $E$  & $0.015$ & $0.013$ & $0.011$  \\
\hline
Compression on $E_{a,d}$  & $0.034$ & $0.037$ & $0.035$  \\
\hline
Decompression on $E$  & $0.09$ & $0.13$ & $0.15$ \\
\hline
Decompression on $E_{a,d}$ & $0.14$ & $0.19$ & $0.20$\\
\hline
\end{tabular}\\
\newline \\
The next table contains the ratios of the average times for point compression and decompression on elliptic curves in short Weierstrass form and twisted Edwards curves. \\
\newline
\textbf{Table 6}.\\
\newline
\begin{tabular}{|l|c|c|c|}
\hline
Bit-length of $|\mathcal{T}_3|$ & $192$ & $224$ & $256$   \\
\hline
Comp on $E$ / Comp on $E_{a,d}$ & $0.441$ & $0.351$ & $0.314$  \\
\hline
Dec on $E$ / Dec on $E_{a,d}$ & $0.643$ & $0.684$ & $0.750$ \\
\hline
\end{tabular}

\subsection{Explicit equations, complexity, and timings for $n=5$}

In this subsection we give explicit equations and perform computations for $n=5$. We estimate the number of operations needed for the computations and present some timings obtained with Magma. We also make comparisons with the method proposed in~\cite{EM2} for elliptic curves in short Weierstrass form. \\

{\bf Point Compression}. Let $P \in \mathcal{T}_5$. By Theorem~\ref{Th_es_poli}, $q_P$ is of the form
$$q_P(x,y)=(1+y)\hat{q_1}(y)+xq_2(y)=(1+y)(a_1y+a_0)+ x(b_2y^2+b_1y+b_0)$$
where $a_0$, $a_1$, $b_0$, $b_1 \in \mathbb{F}_q$, and $b_2\in \mathbb{F}_2$. 
By (\ref{fH_miller})
$$(1+y)h_1 h_2  q_P = \phi_1\phi_2 \phi_3 (a-dy^2)$$
modulo $E_{a,d}$ and up to a nonzero constant factor.
We consider the generic case, where $b_2=1$ and $\phi_i$ is of the form
$$\phi_i(x,y)=p_i(y+1)+x(y+q_i)$$
with $p_i$, $q_i \in \mathbb{F}_{q^5}$, and $i \in \{1,2,3\}$.
Denote by $k_1$ and $k_2$ the $y$-coordinates of $P_1+ P_2$ and $P_1+ P_2 + P_3$, respectively. We have
$$\mathcal{R}(P)=(a_0,a_1,b_0,b_1,1),$$
where
$$\begin{tabular}{lcl}
$a_1$ & $=$ & $k\cdot(d(p_1p_2p_3)+(p_1+p_2+p_3)),$ \\
$a_0$ & $=$ & $k\cdot(3d(p_1p_2p_3)+(p_1q_2+p_1q_3+q_1p_2+q_1p_3+p_2q_3+q_2p_3)+(p_1+p_2+p_3))+$\\
&& $a_1\cdot(k_1+k_2-2),$\\
$b_1$ & $=$ & $k\cdot(d(p_1p_2q_3+p_1p_3q_2+p_2p_3q_1)+2d(p_1p_2+p_1p_3+p_2p_3)+(q_1+q_2+q_3))+$\\
&&$(k_1+k_2-1),$\\
$b_0$ & $=$ & $k\cdot(2d(p_1p_2q_3+p_1p_3q_2+p_2p_3q_1)+(d-a)(p_1p_2+p_1p_3+p_2p_3)+$\\
&& $(q_1q_2+q_1q_3+q_2q_3))-1) + b_1(k_1+k_2-1)+(k_1+k_2-k_1k_2)$,\\
$k$ & $=$ & $(d(p_1p_2+p_1p_3+p_2p_3)+1)^{-1}$.\\
\end{tabular}$$
Computing $\phi_1$, $\phi_2$, and $\phi_3$ takes $2$S+$34$M+$2$I in $\mathbb{F}_{q^5}$. Computing $a_1$, $a_2$, $b_1$, $b_0$ with the formulas above requires $45$M+$1$I in $\mathbb{F}_{q^5}$. So point compression for $n=5$ takes a total of $2$S+$79$M+$3$I in $\mathbb{F}_{q^5}$. 
The method of ~\cite{EM2} for elliptic curves in short Weierstrass form is less expensive, as it takes $3$S+$18$M+$3$I in $\mathbb{F}_{q^5}$.\\

{\bf Point Decompression}. Let $(\alpha_1,\alpha_2,\alpha_3,\alpha_4,b)\in \mathbb{F}_q^4\times\mathbb{F}_2$ and let $P=(u,v)\in\mathcal{T}_5$ such that $\mathcal{R}(P)=(\alpha_1,\alpha_2,\alpha_3,\alpha_4,b)$. In order to decompress $\mathcal{R}(P)$, we look for the roots of
$$Q_P(y)=Q_5y^5+Q_4y^4+Q_3y^3+Q_2y^2+Q_1y+Q_0,$$
where
$$\begin{tabular}{lcl}
$Q_0$ & $=$ & $a\alpha_1^2-\alpha_3^2$ \\
$Q_1$ & $=$ & $a\alpha_1^2+2a\alpha_1\alpha_2+\alpha_3^2-2\alpha_3\alpha_4$\\
$Q_2$ & $=$ & $-d\alpha_1^2+2a\alpha_1\alpha_2+a\alpha_2^2+2\alpha_3\alpha_4-2\alpha_3b-\alpha_4^2$\\
$Q_3$ & $=$ & $-d\alpha_1^2-2d\alpha_1\alpha_2+a\alpha_2^2+2\alpha_3b+\alpha_4^2-2\alpha_4b$\\
$Q_4$ & $=$ & $-2d\alpha_1\alpha_2-d\alpha_2^2+2\alpha_4b-b$.\\
$Q_5$ & $=$ & $-d\alpha_2^2+b$.
\end{tabular}$$
This amounts to solving the system
$$\left\{\begin{array}{rcr}
e_1(y,y^q,\dots,y^{q^4}) & = & -Q_5^{-1}Q_4\\
e_2(y,y^q,\dots,y^{q^4}) & = & Q_5^{-1}Q_3\\
e_3(y,y^q,\dots,y^{q^4}) & = & -Q_5^{-1}Q_2\\
e_4(y,y^q,\dots,y^{q^4}) & = & Q_5^{-1}Q_1\\
e_5(y,y^q,\dots,y^{q^4}) & = & -Q_5^{-1}Q_0
\end{array}\right.$$
where $e_i(y,y^q,\dots,y^{q^4})$ is the $i$-th elementary symmetric polynomial in $y,y^q,\dots,y^{q^4}$.
Computing the constants in the system takes $4$S+$7$M+$1$I in $\mathbb{F}_q$, while solving the system requires $\OO(\log_2{q})$ operations in $\mathbb{F}_q$ following the approach from~\cite{EM2}. Finally, recovering $u$ from $v$ takes $1$S+$5$M+1$I$ in $\mathbb{F}_{q^5}$. The computational cost of point decompression is comparable to that of the decompression algorithm from~\cite{EM2} for elliptic curves in short Weierstrass form.\\
 
In order to estimate of the average time of compression and decompression for groups of different bit-size,
we consider primes $q_1$, $q_2$, and $q_3$ such that $3|q_i-1$ for all $i$, of bit-length $96$, $112$, and $128$,  respectively. For each $q_i$, we consider five pairs of birationally equivalent curves $(E,E_{a,d})$ defined over $\mathbb{F}_{q_i}$, such that the order of $\mathcal{T}_3$ is prime of bit-length respectively $192$, $224$ and $256$.
On each pair of curves we randomly choose $20'000$ pairs of points $(P',P)$ of trace zero which correspond to each other via the birational isomorphism between the curves. For each pair of points, we compute
$\mathcal{R'}(P'), \mathcal{R}(P), \mathcal{R'}^{-1}(\mathcal{R'}(P')), \mathcal{R}^{-1}(\mathcal{R}(P)).$
For each computation, we consider the average time in milliseconds for each curve, and then the averages over the five curves. The average computation times are reported in the table below.\\
\newline
\textbf{Table 7}.\\
\newline
\begin{tabular}{|l|c|c|c|}
\hline
Bit-length of $|\mathcal{T}_5|$ & $192$ & $224$ & $256$   \\
\hline
Compression on $E$  & $1.566$ & $1.725$ & $1.894$ \\
\hline
Compression on $E_{a,d}$  & $1.704$ & $1.868$ & $2.052$ \\
\hline
Decompression on $E$  & $6.10$ & $31.69$ & $36.99$ \\
\hline
Decompression on $E_{a,d}$ & $6.15$ & $31.37$ & $36.59$\\
\hline
\end{tabular}\\
\newline

The next table contains the ratios of the average times for point compression and decompression on elliptic curves in short Weierstrass form and twisted Edwards curves. \\
\newline
\textbf{Table 8}.\\
\newline
\begin{tabular}{|l|c|c|c|}
\hline
Bit-length of $|\mathcal{T}_5|$ & $192$ & $224$ & $256$ \\
\hline
Comp on $E$ / Comp on $E_{a,d}$ & $0.919$ & $0.923$ & $0.923$  \\
\hline
Dec on $E$ / Dec on $E_{a,d}$ & $0.992$ & $1.010$ & $1.011$ \\
\hline
\end{tabular}\\
\newline\newpage
Finally, Table 9 summarizes the number of operations for point compression and decompression. We compare the operation count from this paper with the one for elliptic curves in short Weierstrass form from~\cite{EM2}.\\
\newline
\textbf{Table 9.}\\
\newline
\begin{tabular}{|l|l|}
\hline
Compression, $n=3$, elliptic & $2$S+$6$M+$1$I in $\mathbb{F}_q$ \\
\hline
Compression, $n=3$, Edwards & $1$M+$1$I in $\mathbb{F}_{q^3}$ and  $2$S+$6$M+$1$I in $\mathbb{F}_q$ \\
\hline
Decompression, $n=3$, elliptic & $1$M in $\mathbb{F}_{q^3}$, $5$S+$4$M+$1$I, one square root, two cube roots in $\mathbb{F}_q$  \\
\hline
Decompression, $n=3$, Edwards &  $2$M + $1$I in $\mathbb{F}_{q^3}$, $5$S+$6$M+$2$I, one square root, two cube roots in $\mathbb{F}_q$\\
\hline
Compression, $n=5$, elliptic & $3$S+$18$M+$3$I in $\mathbb{F}_{q^5}$ \\
\hline
Compression, $n=5$, Edwards & $2$S+$79$M+$3$I in $\mathbb{F}_{q^5}$ \\
\hline
Decompression, $n=5$, elliptic & $\OO(\log_2{q})$ operations in $\mathbb{F}_q$, $1$S+$3$M+$1$I in $\mathbb{F}_{q^5}$ \\
\hline
Decompression, $n=5$, Edwards & $\OO(\log_2{q})$ operations in $\mathbb{F}_q$, $1$S+$5$M+$1$I in $\mathbb{F}_{q^5}$\\
\hline
\end{tabular}\\


\begin{thebibliography}{9}

 	

\bibitem{handbook} R. M. Avanzi, H. Cohen, C. Doche, G. Frey, T. Lange, K. Nguyen, F. Vercauteren, {\em Handbook of Elliptic and Hyperelliptic Curve Cryptography}, Discrete Mathematics and Its Applications 34, Chapman \& Hall/CRC (2005).

\bibitem{GeomInt} C. Ar\'ene. T. Lange, M. Naehrig, C. Ritzenthaler,
\emph{Faster Computation of the Tate Pairing}, Journal of Number Theory 131, no. 5 (2011), 842-857.

\bibitem{ac07} R. M. Avanzi, E. Cesena, {\em Trace  zero varieties over  fields  of characteristic  2 for  cryptographic
applications}, Proceedings of the First Symposium on Algebraic Geometry and Its Applications -- SAGA '07 (2007), 188-215.

\bibitem{BL1} D. J. Bernstein, T. Lange,
\emph{Faster addition and doubling on elliptic curves}, Advances in Cryptology - ASIACRYPT 2007, LNCS vol. 4833, Springer-Verlag (2007), 29-50.

\bibitem{BL2} D. J. Bernstein, T. Lange,
\emph{Inverted Edwards Coordinates}, Applied Algebra, Algebraic Algorithms and Error-Correcting Codes, LNCS vol .4851, Springer-Verlag (2007), 20-27.

\bibitem{BLal} D. J. Bernstein, P. Birkner, M. Joye, T. Lange, C. Peters,
\emph{Twisted Edwards Curves}, Progress in Cryptology - AFRICACRYPT 2008, LNCS vol. 5023, Springer-Verlag (2008), 389-405.



\bibitem{magma} W. Bosma, J. Cannon, C.e Playoust, {\em The Magma algebra system. I. The user language}, 
J. Symbolic Comput. 24 (1997), 235��-265. 

\bibitem{cesenathesis} E. Cesena, {\em Trace zero varieties in pairing-based cryptography}, Ph.D. Thesis (2010), 
available at \url{https://ricerca.mat.uniroma3.it/dottorato/Tesi/tesicesena.pdf}.



\bibitem{EdEd} H. M. Edwards,
\emph{A Normal Form for Elliptic Curves}, Bulletin of the American Mathematical Society 44 (2007), 393-422.

\bibitem{Frey} G.Frey,
\emph{Applications of Arithmetical Geometry to Cryptographic Constructions}, Proceedings of the 5th International Conference on Finite Fields and Applications, Springer (1999),128-161.


\bibitem{EM1} E. Gorla, M. Massierer, 
\emph{Point Compression for the Trace Zero Subgroup over a Small Degree Extension Field}, 
Designs, Codes and Cryptography 75, no. 2 (2015), 335-357.

\bibitem{EM2} E. Gorla, M. Massierer,
\emph{An Optimal Representation for the Trace Zero Subgroup}, available at \url{http://arxiv.org/abs/1405.2733}.

\bibitem{langetzero} T. Lange, 
\emph{Trace zero subvarieties of genus $2$ curves for cryptosystem}, Ramanujan Math. Soc. 19, no. 1 (2004) 15-33.


\bibitem{nau99} N. Naumann, {\em Weil-Restriktion abelscher Variet\"aten}, 
Master's thesis (1999), available at \url{http://web.iem.uni-due.de/ag/numbertheory/dissertationen}.



\bibitem{rs09} K. Rubin, A. Silverberg, {\em Using abelian varieties to improve pairing-based cryptography},
Journal of Cryptology 22, no. 3 (2009), 330-364.

\bibitem{SemPol} I. Semaev,
\emph{Summation polynomials and the discrete logarithm problem on elliptic curves}, 
available at \url{http://eprint.iacr.org/2004/013, 2004}.

\bibitem{sil05} A. Silverberg, {\em Compression for Trace Zero Subgroups of Elliptic Curves}, 
Trends in Mathematics 8 (2005), 93-100. 

\bibitem{SemPolEd} J. C. Faug\'ere, P. Gaudry, L. Huot, G. Renault,
\emph{Using Symmetries in the Index Calculus for Elliptic Curves Discrete Logarithm}, 
Journal of Cryptology 27, no. 4 (2014),595-635.



\end{thebibliography}
\end{document}